\documentclass[11pt]{amsart}

\usepackage{enumerate}
\usepackage{amssymb,amsfonts,amsthm}
\usepackage{srcltx}
\usepackage{graphicx}
\usepackage{hyperref}

\usepackage{float}
\usepackage{chngcntr}
\usepackage[OT2,OT1]{fontenc}
\newcommand\cyr{%
\renewcommand\rmdefault{wncyr}%
\renewcommand\sfdefault{wncyss}%
\renewcommand\encodingdefault{OT2}%
\normalfont
\selectfont}
\DeclareTextFontCommand{\textcyr}{\cyr}
\usepackage{color}
\makeatletter
\def\sssub{\@startsection{paragraph}{4}}

\renewcommand\paragraph{\@startsection{paragraph}{4}{\z@}{1.25ex}{0.0001pt}{\normalfont\normalsize\em}}

\makeatother

\numberwithin{paragraph}{subsubsection}

\newcommand{\Z}{\mathbb{Z}}
\newcommand{\Q}{\mathbb{Q}}
\newcommand{\CA}{\mathcal{A}}
\newcommand{\CB}{\mathcal{B}}

\newcommand{\CK}{\mathcal{K}}
\newcommand{\CL}{\mathcal{L}}
\newcommand{\CM}{\mathcal{M}}
\newcommand{\CN}{\mathcal{N}}

\newcommand{\coef}{\mathrm{coef}}

\newcommand{\mo}{{-1}}

\newcommand{\bdy}{\partial}

\newcommand{\prn}[1]{\left( #1 \right)}

\newcommand{\set}[1]{\left\{ #1 \right\}}

\newcommand{\ang}[1]{\left\langle #1 \right\rangle}
\newcommand{\ol}{\overline}
\DeclareMathOperator{\tp}{tp}

\DeclareMathOperator{\Th}{Th}

\usepackage[dvipsnames]{xcolor}

\newtheorem{theorem}{Theorem}[section]

\newtheorem{corollary}[theorem]{Corollary}
\newtheorem{proposition}[theorem]{Proposition}

\newtheorem{lemma}[theorem]{Lemma}

\theoremstyle{definition}
\newtheorem{definition}[theorem]{Definition}

\theoremstyle{remark}
\newtheorem{remark}[theorem]{Remark}

\counterwithin*{case}{theorem}

\newtheoremstyle{example}
    {\dimexpr\topsep/2\relax} 
    {\dimexpr\topsep/2\relax} 
    {}          
    {}          
    {\bfseries} 
    {.}         
    {.5em}      
    {}          
\theoremstyle{example}
\newtheorem{example}[theorem]{Example}

\usepackage[maxbibnames=15, style=alphabetic]{biblatex}
\begin{document}

\title{Non-$\forall$-homogeneity in free groups}

\author{Olga Kharlampovich}
\thanks{Olga Kharlampovich: City University of New York, Graduate Center and Hunter College}
\email[Olga Kharlampovich]{okharlampovich@gmail.com}

\author{Christopher Natoli}
\thanks{Christopher Natoli: City University of New York, Graduate Center}
\email[Christopher Natoli]{chrisnatoli@gmail.com}


\maketitle

\begin{abstract}
    We prove that non-abelian free groups of finite rank at least 3 or of countable rank are not $\forall$-homogeneous. We answer three open questions from Kharlampovich, Myasnikov, and Sklinos regarding whether free groups, finitely generated elementary free groups, and non-abelian limit groups form special kinds of Fra\"iss\'e classes in which embeddings must preserve $\forall$-formulas. We also provide interesting examples of countable non-finitely generated elementary free groups.
\end{abstract}

\section{Introduction}

Perin and Sklinos \cite{PerSkl2012homogeneity} and, independently, Ould Houcine \cite{Oul2011homogeneity} proved that non-abelian free groups have the property that two tuples $\bar a$ and  $\bar b$ realize the same  first-order formulas if and only if there is an automorphism of the group sending $\bar a$ to $\bar b$.  This is a model-theoretic property called ($\aleph_0$-)homogeneity. This added to earlier results from Nies \cite{Nie2003aspects} that showed that the free group on two generators is homogeneous. In fact, Nies showed that the free group on two generators has a stronger property, which we will call $\forall$-homogeneity. A structure $\CM$ is $\forall$-homogenous if for any tuples $\bar a,\bar b\in\CM$ of the same length, if $\bar a$ and $\bar b$ realize the same universal formulas, then there is an automorphism of $\CM$ sending $\bar a$ to $\bar b$. It was an open question whether all non-abelian free groups are $\forall$-homogeneous.

The main result of this paper is providing a counterexample (Example~\ref{example-no-homog}), namely, a non-abelian free group of rank 3 that is not $\forall$-homogeneous, and extending this result to all countable free groups of higher rank.

As an immediate corollary we have that the first-order theory of a non-abelian free group does not have quantifier elimination to boolean combinations of $\forall$-formulas. Notice that  it was shown in \cite{KhaMya2006elementary} and \cite{Sel2006diophantine} that this theory has quantifier elimination down to boolean combinations of $\forall\exists$-formulas.
It is known that there is no quantifier elimination  to $\forall$-formulas because  the theory of a non-abelian free group is not model complete \cite{Per2011elementary}.

 We also build on this example to answer, in the negative, three open questions from Kharlampovich, Myasnikov, and Sklinos in \cite{KhaMyaSkl2020fraisse} regarding special kinds of Fra\"iss\'e classes: whether finitely generated, non-abelian free groups form a $\forall$-Fra\"iss\'e class; whether finitely generated, elementary free groups form a $\forall$-Fra\"iss\'e class, and whether non-abelian limit groups form a strong $\forall$-Fra\"iss\'e class. To compare, \cite{KhaMyaSkl2020fraisse} showed that the class of non-abelian limit groups is a $\forall$-Fra\"iss\'e class and that the class of abelian limit groups, i.e., finitely generated free abelian groups, form a strong $\forall$-Fra\"iss\'e class. Finally, we answer a fourth question from \cite{KhaMyaSkl2020fraisse}, showing that not all countable elementary free groups are obtained as the union of a chain of finitely generated elementary free groups.

\section{Preliminaries}

For the rest of this paper, free groups will mean non-abelian free groups, unless stated otherwise.

\begin{definition}
    Let $\CM$ be a model. Given a tuple $\bar a\in\CM$, the \emph{type} of $\bar a$ in $\CM$ is
    $\tp^\CM(\bar a)=\set{\phi(x):\CM\models\phi(\bar a)}$.
    We say $\CM$ is ($\aleph_0$-)\emph{homogeneous} if for any tuples $\bar a,\bar b\in\CM$ of the same length, $\tp^\CM(\bar a)=\tp^\CM(\bar b)$ implies there is an automorphism of $\CM$ sending $\bar a$ to $\bar b$.
    If $\CM$ is a substructure of $\CN$, then $\CM$ is an \emph{elementary substructure} of $\CN$, denoted $\CM\prec\CN$, if for every $\bar a\in\CM$, we have $\tp^\CM(\bar a)=\tp^\CN(\bar a)$.

    Analogously, if $\bar a\in\CM$, we define its \emph{$\forall$-type} by
    $\tp^\CM_\forall(\bar a)=\set{\phi(x):\text{$\CM\models\phi(\bar a)$, $\phi(x)$ is universal}}$.
    $\CM$ is \emph{$\forall$-homogeneous} if for any tuples $\bar a,\bar b\in\CM$ of the same length, $\tp^\CM_\forall(\bar a)=\tp^\CM_\forall(\bar b)$ implies there is an automorphism of $\CM$ sending $\bar a$ to $\bar b$.  (Note that Ould Houcine calls this latter property $\exists$-homogeneity.) If $\CM$ is a substructure of $\CN$, then $\CM$ is \emph{existentially closed} in $\CN$ if for every $\bar a\in\CM$, we have $\tp_\forall^\CM(\bar a)=\tp_\forall^\CN(\bar a)$.
\end{definition}

Full characterizations of elementary substructures and existentially closed substructures in the context of free groups have been proven. The work by Kharlampovich and Myasnikov \cite{KhaMya2006elementary} and, separately, Sela \cite{Sel2006diophantine} in positively answering Tarski's question of whether non-abelian free groups are elementarily equivalent proved the stronger result that any non-abelian free factor $F_n$ of a free group $F_{n+m}$ is an elementary substructure, provided $n\ge2$. Perin proved the converse, thus characterizing elementary substructures for free groups:
\begin{theorem}\label{perin}
    \cite[Theorem~1.3]{Per2011elementary}
    Let $H$ be a proper subgroup of a finitely generated free group $F$. Then $H$ is an elementary substructure of $F$ if and only if $H$ is a non-abelian free factor of $F$.
\end{theorem}

Kharlampovich, Myasnikov, and Sklinos give a characterization of existentially closed subgroups of limit groups (a class of groups that includes free groups) in terms a construction called extensions of centralizers.

\begin{definition}
    An \emph{extension of a centralizer} of a group $G$ is a group $\ang{G,t\mid[C_G(u),t]=1}$ where $u$ is some fixed element in $G$ and $t$ is a new letter. If $G=G_0<\cdots<G_n$ and each $G_{k+1}$ is an extension of a centralizer of $G_k$, i.e., $G_{k+1}=\ang{G_k,t_k\mid[C_{G_k}(u_k),t_k]=1}$ where $u_k\in G_k$, we call $G_n$ a \emph{finite iterated extension of centralizers} over $G$.
\end{definition}

Recall that a group $G$ is called {\it fully residually
free} (or {\it freely discriminated}, or {\it $\omega$-residually
free}) if for any finite subset of non-trivial elements $g_1, \ldots,
g_n \in G$ there exists a homomorphism $\phi$ of $G$ into a free
group $F$, such that $\phi(g_i) \neq 1$ for $i = 1, \ldots, n$.   $L$  is fully residually free if  and only if $\Th_\forall(L)=\Th_\forall(F)$ where $F$ is a  free group and $\Th_\forall$ is the universal theory of a structure. Finitely generated fully residually free groups are also known as {\em limit groups}.

\begin{lemma}\label{ex-closed-ice}
    \cite[Theorem~3.6 and Lemma~3.7]{KhaMyaSkl2020fraisse}
    Let $L,M$ be  limit groups with $L\le M$. Then $L$ is existentially closed in $M$ if and only if there is a finite iterated centralizer extension $L_n$ of $L$ such that $M\le L_n$.
\end{lemma}

\begin{remark}\label{remark-ice}
    For embeddings into a sequence of centralizer extensions, we can without loss consider the sequence to be a mixture of centralizer extensions and free products with free groups. I.e., suppose $L$ is a non-abelian limit group and $M$ embeds in a sequence $L=L_0<\cdots<L_n$ where for all $k$, $L_{k+1}$ is a centralizer extension of $L_k$ or a free product of $L_k$ with a countable free group. Then $M$ can be obtained as a subgroup of a finite iterated centralizer extension over $L$. Indeed, if $L_{k+1}=L_k*\ang{x_1,\ldots,x_m}$ where $x_i$ are new letters, then $L_{k+1}$ embeds in the centralizer extension $\ang{L_k,t\mid[C_{L_k}(u),t]=1}$ by mapping $x_i$ to $t^igt^igt^i$ where $g\in L_k-C_{L_k}(u)$.
\end{remark}

\cite{KhaMyaSkl2020fraisse} define two special kinds of Fra\"iss\'e classes, $\forall$-Fra\"iss\'e classes and strong $\forall$-Fra\"iss\'e classes. A \emph{$\forall$-embedding} of a structure into another (or a \emph{partial $\forall$-embedding} of a tuple into a structure), denoted with $\to_\forall$, is an embedding that preserves $\forall$-formulas. Note that the inclusion map from some model $\CA$ into another model $\CB$ is a $\forall$-embedding if and only if $\CA$ is existentially closed in $\CB$. Fix a language $\CL$.
\begin{definition}
    Let $\CK$ be a countable (with respect to isomorphism types) non-empty class of finitely generated $\CL$-structures with the following properties:
    \begin{itemize}
        \item (IP) the class $\CK$ is closed under isomorphisms;
        \item ($\forall$-HP) the class $\CK$ is closed under finitely generated $\forall$-substructures (i.e., existentially closed substructures);
        \item ($\forall$-JEP) if $\CA_1,\CA_2$ are in $\CK$, then there are $\CB$ in $\CK$ and $\forall$-embeddings $f_i:\CA_i\rightarrow_{\forall} \CB$ for $i\leq 2$;
        \item ($\forall$-AP) if $\CA_0,\CA_1,\CA_2$ are in $\CK$ and $f_i:\CA_0\rightarrow_{\forall} \CA_i$ for $i\leq 2$ are $\forall$-embeddings, then there are $\CB$ in $\CK$ and $\forall$-embeddings $g_i:\CA_i\rightarrow_{\forall} \CB$ for $i\leq 2$ with $g_1\circ f_1=g_2\circ f_2$.  
    \end{itemize}
    Then $\CK$ is a \emph{universal Fra\"{i}ss\'{e} class} or for short a \emph{$\forall$-Fra\"{i}ss\'{e} class}.
    If in addition, $\CK$ satisfies
    \begin{itemize}
        \item (strong $\forall$-AP) if $\CA_0,\CA_1,\CA_2$ are in $\CK$ and $f_i:\bar a\rightarrow_{\forall} \CA_i$ for $i\leq 2$ are partial $\forall$-embeddings of some tuple $\bar a\in\CA_0$, then there are $\CB$ in $\CK$ and $\forall$-embeddings $g_i:\CA_i\rightarrow_{\forall} \CB$ for $i\leq 2$ with $g_1\circ f_1(\bar a)=g_2\circ f_2(\bar a)$,
    \end{itemize}
    then $\CK$ is a \emph{strong universal Fra\"{i}ss\'{e} class} or for short a \emph{strong $\forall$-Fra\"{i}ss\'{e} class}.
\end{definition}

\section{Counterexamples}

\subsection{$\forall$-homogeneity}

We first describe a free group $M$ of rank 4 that is not $\forall$-homogeneous and then a free group $M_3$ of rank 3 that is not $\forall$-homogeneous.

\medskip

\begin{example}\label{example-no-homog}
    Let
    $$
    L=\ang{a,b}
    \quad\le\quad
    L_1=L*\ang{x}
    \quad\le\quad
    L_2=\ang{L_1,t\mid[u,t]=1},
    $$
    where $x$ is a new letter and $u=x^2(bx^{2n})^m$. We construct $M$ as an amalgamated product $M_1*_{u=u^t}M_2$ living in $L_2$, where $M_1=\ang{a,b,x^2}$ and $M_2=\ang{b^t,x^t}$. Nielsen transformations show that $\set{a,bx^{2n},x^2(bx^{2n})^m}$ is a basis for $M_1$. Then
    $$
    M
    =M_1\underset{u=u^t}{*}M_2
    =\ang{a,bx^{2n},u}\underset{u=u^t}{*}\ang{b^t,x^t}
    =\ang{a,bx^{2n}}*\ang{b^t,x^t}.
    $$
    So $M$ is a free subgroup of $L_2$ containing $L$, and by Lemma~\ref{ex-closed-ice}, $L$ is existentially closed in $M$. But $L$ is not a free factor of $M$, since $b$ is not primitive in $M$.
\end{example}

Notice that $\tp_\forall^L(a,b)=\tp_\forall^M(a,b)$.  Also, the tuple $(a,bx^{2n})$ has the same $\forall$-type in $M$ as it does in $\langle a, bx^{2n}\rangle$, because this group is a free factor in $M$ and therefore, by Theorem~\ref{perin}, is elementarily embedded in $M$. Moreover, $L$ and $\langle a, bx^{2n}\rangle$ are isomorphic, so $\tp_\forall^M(a,b)=\tp_\forall^M(a,bx^{2n})$. But there is no automorphism of $M$ sending $(a,b)$ to $(a,bx^{2n})$ because $L$ is not a free factor, so $M$ is not $\forall$-homogeneous. 

Now let $M_3$ be the subgroup of $M$ generated by $bx^{2n}, b^t,x^t.$ Since $\tp_\forall^M(a,b)=\tp_\forall^M(a,bx^{2n})$, we have  $\tp_\forall^{M}(b)=\tp_\forall^{M}(bx^{2n}).$  Since $M_3$ is a non-abelian free factor of $M$ containing both $b$ and $bx^{2n}$, we have $\tp_\forall^{M_3}(b)=\tp_\forall^{M_3}(bx^{2n})$. But $bx^{2n}$ is primitive in $M_3$ and $b$ is not, so they cannot be in the same automorphic orbit. 

An alternative proof that $M$ is not $\forall$-homogeneous follows from Ould Houcine's characterization of $\forall$-homogeneity for finitely generated free groups:
\begin{proposition}\label{ould-houcine-characterization}
    \cite[Proposition~4.10]{Oul2011homogeneity}
    Let $F$ be a finitely generated free group. Then $F$ is $\forall$-homogeneous if and only if $F$ satisfies the following two conditions:
    \begin{enumerate}
        \item If a tuple $\bar a\in F$ is a power of a primitive element (i.e., there is a single primitive element $x\in F$ such that $\bar a\in\ang{x}$) and if $\tp^F_\forall(\bar a)=\tp^F_\forall(\bar b)$, then $\bar b$ is a power of a primitive element.
        \item Every existentially closed subgroup of $F$ is a free factor.
    \end{enumerate}
\end{proposition}
\noindent $L$ is an example of an existentially closed subgroup of $M$ that is not a free factor, hence $M$ is not $\forall$-homogeneous. 

\begin{theorem}
    Free groups of finite rank at least 3 or of countable rank are not $\forall$-homogeneous.
\end{theorem}

\begin{proof}

    Let $L,M,M_3$ be as in Example~\ref{example-no-homog}. We have shown that $M_3$ and $M$ are free groups of ranks 3 and 4, respectively, and neither are $\forall$-homogeneous. Suppose $F$ is a finitely generated free group of rank greater than 4, and canonically embed $M$ into $F$. Since this embedding is elementary, $L$ is existentially closed in $F$. Suppose by way of contradiction that $F$ is $\forall$-homogeneous. Then by Proposition~\ref{ould-houcine-characterization}, $L$ is a free factor of $F$, say $F=L*K$. By Bass-Serre theory, we can write
    $$
    M=L*(L^{x_1}\cap M)*\cdots*(L^{x_p}\cap M)*(K^{y_1}\cap M)*\cdots*(K^{y_q}\cap M)*F',
    $$
    where $x_i,y_i\in F$ and $F'$ is some free group. But then $L$ is a free factor of $M$.

    Let $F_\omega$ be a free group of rank $\omega$, and embed $M$ in $F_\omega$ canonically. Choose tuples $\bar a,\bar b\in M$ such that $\tp_\forall^M(\bar a)=\tp_\forall^M(\bar b)$ but there is no automorphism of $M$ sending $\bar a$ to $\bar b$. It is a result from model theory (see for example \cite[Proposition~2.3.11]{Mar2002model}) that any structure in an elementary chain is an elementary substructure of the union of the chain. In particular, if we let $F_i$ denote the free group on $i$ generators, then $F_2\prec F_3\prec M\prec F_5\prec\cdots$ form an elementary chain, so $M\prec\bigcup_{i<\omega}F_i=F_\omega$. Then $\tp_\forall^{F_\omega}(\bar a)=\tp_\forall^{F_\omega}(\bar b)$. If $F_\omega$ were $\forall$-homogeneous, then there would be an automorphism of $F_\omega$ sending $\bar a$ to $\bar b$. Then $\tp^{F_\omega}(\bar a)=\tp^{F_\omega}(\bar b)$, so $\tp^M(\bar a)=\tp^M(\bar b)$. Since $M$ is homogeneous, there is an automorphism of $M$ carrying $\bar a$ to $\bar b$.
\end{proof}

\begin{corollary}
    The first-order theory of a non-abelian free group does not have quantifier elimination to boolean combinations of $\forall$-formulas.
\end{corollary}

\subsection{$\forall$-AP} 

In this section we will prove the following:
\begin{theorem}\label{th:free}
    The class of finitely generated free groups is not a $\forall$-Fra\"iss\'e class.
\end{theorem}

In the next example we double Example \ref{example-no-homog} to obtain two finitely generated free groups that cannot be amalgamated to satisfy $\forall$-AP. 

\medskip

\begin{example}
    Following Example~\ref{example-no-homog}, we let $L=\ang{a,b}$ be a common subgroup of $H$ and $K$, where
    \begin{align*}
        H&=\ang{a,h,\tilde b,\tilde x}
        &K&=\ang{a,k,\hat b,\hat y}\\
        b&\mapsto h\prn{\tilde x^8(\tilde b\tilde x^{8n})^{m}h^{-m}}^{-n}
        &b&\mapsto k\prn{\hat y^7(\hat b\hat y^{7p})^{q}k^{-q}}^{-p},
    \end{align*}
   $n,m,p,q$ are sufficiently large, and $m,q$ are even. Here, we change the exponents, $\tilde b$ plays the role of $b^t$ from Example~\ref{example-no-homog}, $\tilde x$ plays the role of $x^t$, $h$ plays the role of $bx^{8n}$, and the embedding of $b$ into $H$ is based on the equation $u=u^t$ from $M$. Similarly for $K$, with a new letter $y$ replacing $x$.
\end{example}

We require a few lemmas to show that $H$ and $K$ cannot be amalgamated over $L$ into a free group to satisfy the $\forall$-AP. The first follows from \cite[Lemma 4]{LynSch1962equation}.
\begin{lemma}\label{lynschu}
    Suppose $F$ is a finitely generated free group. Fix a basis of $F$, and let $u$ and $v$ be cyclically reduced words in that basis. Suppose $w$ is a subword of $u^{n_1}$ and a subword of $v^{n_2}$ where $n_1,n_2\in\Z$, and suppose that $|w|\ge|u|+|v|$. Then there exist $a_1,a_2$ cyclic shifts of one another such that $u=a_1^{k_1}$ and $v=a_2^{k_2}$ for some $k_1,k_2\in\Z$. In particular, $u$ commutes with a conjugate of $v$.
\end{lemma}

\begin{definition}
    An equation $E$ is \emph{quadratic} if each variable in $E$ occurs exactly twice. Let $F(A)$ denote the free group with finite basis $A$, and let $(A\cup A^{-1})^*$ be the set of all words (not necessarily reduced) with alphabet $A$. A quadratic equation $E$ with variables $\{x_i,y_i,z_j\}$ and non-trivial coefficients $\{C_j,C\} \in F(A)$ is said to be in \emph{standard form} if its coefficients are expressed as freely and cyclically reduced words in $(A\cup A^{-1})^*$ and $E$ has either the form
\begin{equation}\label{eqn:orientable}
    \left( \prod_{i=1}^g [x_i,y_i]\right)\left(\prod_{j=1}^{m_\coef-1} z_j^\mo C_j z_j\right) C = 1 \quad\text{or}\quad
    \left(\prod_{i=1}^g [x_i,y_i]\right)C=1
\end{equation}
where $[x,y]=x^\mo y^\mo xy$, in which case we say it is \emph{orientable}, or it has the form
\begin{equation} \label{eqn:non-orientable}
    \left(\prod_{i=1}^g x_i^2\right) \left(\prod_{j=1}^{m_\coef-1} z_j^\mo C_j z_j\right) C = 1 \quad\text{or}\quad
    \left(\prod_{i=1}^g x_i^2\right)C=1
\end{equation}
in which case we say it is non-orientable. The \emph{genus} of a quadratic equation is the number $g$ in Equations~(\ref{eqn:orientable}) and (\ref{eqn:non-orientable}) and $m_\coef$ is the number of coefficients. If $g=0$ then we will define $E$ to be orientable. If $E$ is a quadratic equation we define its \emph{reduced Euler characteristic}, $\ol{\chi}$ as follows:
\[\ol{\chi}(E) =
    \begin{cases}
        2-2g &\text{if $E$ is orientable} \\
        2-g  &\text{if $E$ is not orientable}.
    \end{cases}
\]
\end{definition}

For example, $C_1u^{-1}C_2uv^{-1}C_3v=1$ is an orientable quadratic equation in standard form with variables $u,v$, coefficients $C_1,C_2,C_3$, genus $g=0$, and $m_\coef=3$. Similarly, $C_1v^{-1}C_2v=1$ is an orientable quadratic equation in standard form in a single variable $v$, with coefficients $C_1,C_2$, genus $g=0$, and $m_\coef=2$.

\begin{lemma}\label{Ol}
    \cite{Ols1989diagrams} or \cite[Theorem~4]{KhaVdo2012linear}
    Let $E$ be a quadratic equation in standard form over $F(A)$. If either $g=0$ and $m_\coef=2$, or $E$ is non-orientable and $g=m_\coef=1$, then we set $N=1$. Otherwise we set $N=3(m_\coef-\ol{\chi}(E))$. If $E$ has a solution, then for some $n \leq N$,
  \begin{itemize}
  \item[(i)] there is a set $P = \{p_1,\ldots p_n\}$ of variables and a collection of discs $D_1,\ldots, D_{m_\coef}$ such that
  \item[(ii)] the boundaries of these discs are circular 1-complexes with directed and labelled edges such that each edge has a label in $P$ and each $p_j \in P$ occurs exactly twice in the union of boundaries;
  \item[(iii)] if we glue the discs together by edges with the same label, respecting the edge orientations, then we will have a collection $\Sigma_0,\ldots,\Sigma_l$ of closed surfaces and the following inequalities: if $E$ is orientable then each $\Sigma_i$ is orientable and \[\biggl(\sum_{i=0}^{l} \chi(\Sigma_i)\biggr) - 2l \geq \ol{\chi}(E);\] if $E$ is non-orientable either at least one $\Sigma_i$ is non-orientable and \[\biggl(\sum_{i=0}^{l} \chi(\Sigma_i)\biggr) - 2l \geq \ol{\chi}(E)\] or, each $\Sigma_i$ is orientable and \[\biggl(\sum_{i=0}^{l} \chi(\Sigma_i)\biggr) - 2l \geq \ol{\chi}(E) +2 ;\]
  \item[(iv)] there is a mapping $\ol{\psi}:P \rightarrow (A\cup A^\mo)^*$ such that upon substitution, the coefficients $C_1,\ldots,C_{m_\coef}$ can be read without cancellations around the boundaries of $D_1,\ldots,D_{m_\coef}$, respectively.
  \end{itemize}
\end{lemma}

\begin{lemma} \label{imposs}
    Suppose $a,b,c$ are reduced words in a finitely generated free group $F$ and $a$ is cyclically reduced. Let $m\ge 9$ and $j\in\set{7,8}$. If $|a^mb^mc^j|<|a|$, then one of $a,b,c$ commutes with a conjugate of another or an inverse of another.
\end{lemma}

\begin{proof}
Let $d=a^mb^mc^j$, and let $b=u^{-1}b_0u$ and $c=v^{-1}c_0v$, where $b_0$ and $c_0$ are cyclically reduced. Then we have
$$
d^{-1}a^mu^{-1}b_0^{m}uv^{-1}c_0^j v=1,
$$
which must have a solution in $F$. We apply Lemma~\ref{Ol} to this equation in variables $u,v$. Here, $N=m_\coef=3$. So there are three discs $D_1,D_2,D_3$ with the words $d^{-1}a^m,b^m,c^j$ on the boundaries, and there are two possibilities for $P$, namely, $P=\set{p_1,p_2}$ or $P=\set{p_1,p_2,p_3}$.

Suppose $P=\set{p_1,p_2}$. If any disc is labeled by $p_ip_i$ for some $i$, then that disc is nonorientable, contradicting Lemma~\ref{Ol}(iii). So the only possibility for labeling the boundaries of the discs (up to reordering the discs) is that $p_1$ labels the entirety of $\bdy D_1$, $p_2$ labels the entirety of $\bdy D_2$, and $p_1p_2$ labels $\bdy D_2$.
Suppose $\ol\psi$ in Lemma~\ref{Ol}(iv) sends $p_1$ to a cyclic permutation of $d^{-1}a^m$, $p_2$ to a cyclic permutation of $b_0^m$, and $p_1p_2$ to a cyclic permutation of $c_0^j$.

If there is no cancellation in the $ad^{-1}a$ segment of the cyclic word $d^{-1}a^m$, then $p_1=p_{11}d^{-1}p_{12}$, where $p_{11}=a_1a^{m_1}$, $p_{12}=a^{m_2}a_2$, $m_1+m_2=m-1$, and $a_2a_1=a$.  If there is cancellation in the $ad^{-1}a$ segment, then we can rewrite the cyclic word $d^{-1}a^m$ as $d'a_0^{m-1}a_0'$ where $d'$ is a subword of $d^{-1}$, $a_0$ is a cyclic permutation of $a$, and $a_0'$ is an initial segment of $a_0$. Then $p_1=p_{11}d'p_{12}$, where $p_{11}=a_1a_0^{m_1}a_0'$, $p_{12}=a_0^{m_2}a_2$, $m_1+m_2=m-2$, and $a_2a_1=a_0$.  In either case, $p_{11}$ and $p_{12}$ are subwords of $a^m$, as well as subwords of $c_0^j$. Also, $|p_1|<|p_{11}|+|p_{12}|+|a|$.

By Lemma~\ref{lynschu}, either $a$ commutes with a conjugate of $c_0$, in which case we're done, or both $|p_{11}|<|a|+|c_0|$ and $|p_{12}|<|a|+|c_0|$. Similarly, either $b_0$ commutes with a conjugate of $c_0$ or $|p_2|<|b_0|+|c_0|$. If neither $b_0$ nor $a$ commutes with a conjugate of $c_0$, then 
\begin{align*}
|p_{11}|+|p_{12}|+|p_2| &< 2|a|+|b_0|+3|c_0|\\
|p_1|+|p_2|&< 3|a|+|b_0|+3|c_0|.
\end{align*}
But this contradicts
$$
(m-1)|a|+m|b_0|+j|c_0| < |\bdy D_1|+|\bdy D_2|+|\bdy D_3| = 2|p_1|+2|p_2|,
$$
since $|\bdy D_1|+|\bdy D_2|+|\bdy D_3| = 2|p_1|+2|p_2|$.
Any other choices of $\ol\psi$ are analogous.

Suppose $P=\set{p_1,p_2,p_3}$ and $\bdy D_1$ is labeled by $p_2p_3$, $\bdy D_2$ by $p_1p_3$, and $\bdy D_3$ by $p_1p_2$. Suppose $\ol\psi$ sends $p_1p_2$ to a cyclic permutation of $d^{-1}a^m$, $p_2p_3$ to a cyclic permutation of $b_0^m$, and $p_1p_3$ to a cyclic permutation of $c_0^j$. If $p_1$ covers $ad^{-1}a$, then let $p_1=p_{11}d_1p_{12}$ where $p_{11}$ and $p_{12}$ are  subwords of $a^m$ and $|p_1|<|p_{11}|+|p_{12}|+|a|$. Again, by Lemma~\ref{lynschu}, we have either
\begin{itemize}
    \item $a$ commutes with a conjugate of $b_0$, $b_0$ commutes with a conjugate of $c_0$, or $a$ commutes with a conjugate of $c_0$, in which case we're done; or
    \item $|p_{11}|, |p_{12}|<|a|+|c_0|$ and $|p_2|<|a|+|b_0|$ and $|p_3|<|b_0|+|c_0|$.
\end{itemize}
If the latter, then 
$$
|p_1|+|p_2|+|p_3|<|p_{11}|+|p_{12}|+|a|+|p_2|+|p_3| < 4|a|+2|b|+3|c_0|,
$$
which contradicts
$$
(m-1)|a|+m|b|+j|c_0| < |\bdy D_1|+|\bdy D_2|+|\bdy D_3| = 2|p_1|+2|p_2|+2|p_3|.
$$
Any other labelings of the boundaries or choices of $\ol\psi$ are similar.
\end{proof} 

\begin{proposition} \label{AP}
    There is no finitely generated free group satisfying the $\forall$-AP with respect to $L$, $H$, and $K$.
\end{proposition}

\begin{proof}
Suppose there is such a free group $F$, i.e., suppose $H$ and $K$ $\forall$-embed into $F$ such that the embedding of $L$ along $H\hookrightarrow F$ equals the embedding of $L$ along $K\hookrightarrow F$. Then $F$ must be a quotient of 
\begin{align*}
G
&=H\underset{L}*K\\
&=\ang{a_G,h_G,\tilde b_G,\tilde x_G,k_G,\hat b_G,\hat y_G \;\bigg|\; h_G\prn{\tilde x_G^8(\tilde b_G\tilde x_G^{8n})^{m}h_G^{-m}}^{-n}=k_G\prn{\hat y_G^7(\hat b_G\hat y_G^{7p})^{q}k_G^{-q}}^{-p}},
\end{align*}
where the copies of generators of $H$ and $K$ in $G$ are denoted with subscripts. Denote the images in $F$ of the generators in $G$ without the subscripts, e.g., the image of $h_G\in G$ under the quotient map is a word $h\in F$. The embeddings of $H$ and $K$ into $F$ must commute with their embeddings into $G$ composed with the quotient map, e.g., $h\in H$ is mapped to $h\in F$. So we treat $H$ and $K$ as subgroups of $F$.

Then in $F$ we have the equation
$$
h\prn{h^m(\tilde x^{-8n}\tilde b^{-1})^{m}\tilde x^{-8}}^n\prn{\hat y^7(\hat b\hat y^{7p})^{q}k^{-q}}^{p}k^{-1}=1.
$$
Considering $\tilde b,\tilde x,h$ as coefficients (i.e., parameters) in $H$, we have for all $s>1$,
$$ 
H\models\forall z\prn{\tilde x^8(\tilde b\tilde x^{8n})^{m}h^{-m}\neq z^s}.
$$
Since $H$ is existentially closed in $F$, the same sentence holds in $F$, i.e., $\tilde x^8(\tilde b\tilde x^{8n})^{m}h^{-m}$ is not a proper power in $F$. Similarly $\hat y^7(\hat b\hat y^{7p})^{q}k^{-q},a,h,k,\tilde b,\tilde x,\hat b,\hat y$ are not proper powers in $F$. 

We can assume (changing the values of  $\tilde x, \tilde b, h, \hat b, \hat y, k$ by conjugation, if necessary) that the reduced word in $F$ obtained from $h^m(\tilde x^{-8n}\tilde b^{-1})^{m}\tilde x^{-8}$ is cyclically reduced. We can also assume that the reduced form of $(\hat y^7(\hat b\hat y^{7p})^{q}k^{-q})^pk^{-1}$ is $v^{-1}(\bar y^7(\bar b\bar y^{7p})^{q}\bar k^{-q})^p\bar k^{-1}v$ for some $v$, where $\bar y,\bar b,\bar k$ are conjugates of $\hat y,\hat b,\hat k$ by $v$ and the reduced word obtained from $\bar y^7(\bar b\bar y^{7p})^{q}\bar k^{-q}$ is cyclically reduced.
Finally, we can assume that $h$ and $\bar k$ are cyclically reduced, by changing $\tilde x,\tilde b,\bar y,\bar b,v$ if necessary.

Now we have 
\begin{equation}\label{elf}
  h\prn{h^m(\tilde x^{-8n}\tilde b^{-1})^{m}\tilde x^{-8}}^nv^{-1}\Big(\bar y^7(\bar b\bar y^{7p})^{q}\bar k^{-q}\Big)^{p}\bar k^{-1}v=1.
\end{equation} 
We apply Lemma~\ref{Ol} to this equation in variable $v$. Here, $N=1$ and $m_\coef=2$, so we have a single variable $p_1$ that is the label of both discs $D_1$ and $D_2$. So $p_1$ equals a cyclic permutation of $h\big(h^m(\tilde x^{-8n}\tilde b^{-1})^{m}\tilde x^{-8}\big)^n$ and also equals a cyclic permutation of $\prn{\bar y^7(\bar b\bar y^{7p})^{q}\bar k^{-q}}^{p}\bar k^{-1}$ or its inverse. So, without loss of generality, $h\big(h^m(\tilde x^{-8n}\tilde b^{-1})^{m}\tilde x^{-8}\big)^n$ is a cyclic permutation of $\prn{\bar y^7(\bar b\bar y^{7p})^{q}\bar k^{-q}}^{p}\bar k^{-1}$. Then we can write $hw_1=w_2\bar k^{-1}w_3$, where $w_1=\big(h^m(\tilde x^{-8n}\tilde b^{-1})^{m}\tilde x^{-8}\big)^n$ and $w_2,w_3$ are subwords of $\prn{\bar y^7(\bar b\bar y^{7p})^{q}\bar k^{-q}}^{p}$.

Note that if $|h|>|h^m(\tilde x^{-8n}\tilde b^{-1})^{m}\tilde x^{-8}|$, then by Lemma~\ref{imposs}, one of $h,\tilde x^{-8n}\tilde b^{-1},\tilde x$ commutes with a conjugate of another or an inverse of another. But if $F\models\exists z[h_1,h_2^z]=1$ for $h_1,h_2\in\{h^{\pm1},(\tilde x^{-4n}\tilde b^{-1})^{\pm1},\tilde x^{\pm1}\}$, then $H\models\exists z[h_1,h_2^z]=1$, contradicting the fact that $H$ is freely generated by $a,\tilde x,\tilde b,h$. So $|h|\le|h^m(\tilde x^{-8n}\tilde b^{-1})^{m}\tilde x^{-8}|$ and, similarly, $|\bar k|\le|{\bar y^7(\bar b\bar y^{7p})^{q}\bar k^{-q}}|$.

If $w_3\ne1$, then $w_3$ is a common subword of $\big(h^m(\tilde x^{-8n}\tilde b^{-1})^{m}\tilde x^{-8}\big)^n$ and $\prn{\bar y^7(\bar b\bar y^{7p})^{q}\bar k^{-q}}^{p}$. Otherwise, we have $hw_1=w_2\bar k^{-1}$. Since $|h|\le|h^m(\tilde x^{-8n}\tilde b^{-1})^{m}\tilde x^{-8}|$ and $|\bar k|\le|{\bar y^7(\bar b\bar y^{7p})^{q}\bar k^{-q}}|$, we can choose $w_1',w_2'$ such that $hw_1=hw_1'\bar k^{-1}$ and $w_2\bar k^{-1}=hw_2'\bar k^{-1}$. Then $w_1'=w_2'$, i.e., $\big(h^m(\tilde x^{-8n}\tilde b^{-1})^{m}\tilde x^{-8}\big)^n$ and $\prn{\bar y^7(\bar b\bar y^{7p})^{q}\bar k^{-q}}^{p}$ have a common subword.

Therefore by Lemma~\ref{lynschu}, $h^m(\tilde x^{-8n}\tilde b^{-1})^{m}\tilde x^{-8}$ must equal a conjugate of ${\bar y^7(\bar b\bar y^{7p})^{q}\bar k^{-q}}$.
Then
$$
K\models\exists\check x,\check b,\check h\prn{\check x^8(\check b\check x^{8n})^{m}\check h^{-m}=\hat y^7(\hat b\hat y^{7p})^{q} k^{-q}}.
$$ 
But this equation does not have a solution in $K$, since in general, a free group $F(e_1,e_2,e_3)$ cannot have a solution to the equation $x^8y^{m}z^{-m}=e_1^7e_2^{q}e_3^{-q}$ where $m,q$ are even, because it is impossible for the exponential sum of $e_1$ in the left-hand side to be 7.
 \end{proof}

Now Theorem \ref{th:free} follows from Proposition \ref{AP}.

The proof of Proposition \ref{AP}  can also be extended to finitely generated elementary free groups, i.e., groups that model the common theory of non-abelian free groups.

\begin{theorem}
    The class of finitely generated elementary free groups is not a $\forall$-Fra\"iss\'e class.
\end{theorem}

\begin{proof}
    Let $L,H,K$ be as in the proof of Proposition~\ref{AP}.  Suppose an elementary free group $E$ satisfies the $\forall$-AP with respect to $L,H,K$. The proof of Proposition~\ref{AP} shows that any free group $F$ models the following $\forall\exists$-sentence without parameters: For any values of $h,\tilde x,\tilde b,\bar k,\bar y,\bar b,v$ that solve Equation~(\ref{elf}), there exists $u\in F$ such that $[x^u,y]=1$ for some $x,y\in S_1$ or $x,y\in S_2$ or $x,y\in S_3$, where
    $$
    S_1=\set{h,\tilde x,\tilde b\tilde x^{8n}},\quad
    S_2=\set{\bar k, \bar y, \bar b\bar y^{7p}},\quad
    S_3=\set{h^m(\tilde x^{-8n}\tilde b^{-1})^{m}\tilde x^{-8}, \prn{\bar y^7(\bar b\bar y^{7p})^{q}\bar k^{-q}}}
    $$
    and $x\ne y$. Then $E$ models the same sentence. However, plugging in the words $h,\tilde x,\tilde b,\bar k,\bar y,\bar b\in E$ given by the $\forall$-embeddings of $H$ and $K$ into $E$ results the same contradictions as in Proposition~\ref{AP}.
\end{proof}

\subsection{Strong $\forall$-AP}

We again modify Example~\ref{example-no-homog}, this time to show that non-abelian limit groups do not form a strong $\forall$-Fra\"iss\'e class. 
\medskip

\begin{example}
    Let $L_0=\ang{b,x}$ and $u_1=x^2(bx^{2n})^m$. Define a single centralizer extension $L_1=\ang{L_0,t_1\mid[u_1,t_1]=1}$ with a subgroup $H=\langle h,\tilde b,\tilde x\rangle$ where $\tilde b=b^{t_1}$, $\tilde x=x^{t_1}$, and $h=bx^{2n}$. Note $H$ contains both $b$ and $x^4$ as 
    $$
    b=h\prn{\tilde x^2(\tilde b\tilde x^{2n})^mh^{-m}}^{-n}
    \qquad
    x^4=\prn{\tilde x^2(\tilde b\tilde x^{2n})^mh^{-m}}^2.
    $$
    Analogously, let $u_2=x^4(bx^{4p})^q$ and define another centralizer extension of $L_0$ as $L_2=\ang{L_0,t_2\mid[u_2,t_2]=1}$. Let $K=\langle k,\hat b,\hat x\rangle$ where $\hat b=b^{t_2}$, $\hat x=x^{t_2}$, and $k=bx^{4p}$. $K$ contains both $b$ and $x^4$ as 
    $$
    b=k\prn{\hat x^4(\hat b\hat x^{4p})^qk^{-q}}^{-p}
    \qquad
    x^4=\hat x^4(\hat b\hat x^{4p})^qk^{-q}.
    $$
  
    Note that the inclusions of the tuple $(b,x^4)$ into $H$ and $K$ are partial $\forall$-embeddings. Indeed, if $\phi$ is a quantifier-free formula and $L_0\models\forall\bar y\,\phi(\bar y,b,x^4)$, then by Lemma~\ref{ex-closed-ice}, we have $L_1\models\forall\bar y\,\phi(\bar y,b,x^4)$. Since $H\le L_1$, we have $H\models\forall\bar y\,\phi(\bar y,b,x^4)$. Similarly for $K$.

    Let
    \begin{align*}
    G
    &=H\underset{\ang{b,x^4}}*K\\
    &=\ang{h,\tilde b,\tilde x,k,\hat b,\hat x \ \Bigg|
        \begin{array}{rl}
            h\big(\tilde x^2(\tilde b\tilde x^{2n})^mh^{-m}\big)^{-n}&=k\big(\hat x^4(\hat b\hat x^{4p})^qk^{-q}\big)^{-p}\\
            \big(\tilde x^2(\tilde b\tilde x^{2n})^mh^{-m}\big)^2&=\hat x^4(\hat b\hat x^{4p})^qk^{-q}
        \end{array}}.
    \end{align*}
    Suppose $M$ is a limit group satisfying the strong $\forall$-AP with respect to $L_0,H,K$ and the tuple $(b,x^4)$. Then $M$ must be a quotient of $G$. From the second relation in $G$, we have
    $M\models\exists u\prn{u^2=\hat x^4(\hat b\hat x^{4p})^qk^{-q}}$.
    So $K$ models the same sentence, which is a contradiction.
\end{example}

\begin{theorem}
    The class of non-abelian limit groups is not a strong $\forall$-Fra\"iss\'e class.
\end{theorem}

\subsection{Finite iterated centralizer extensions and free factors}

We prove a result of independent interest, characterizing free factors of free groups in terms of a restricted kind of finite iterated centralizer extensions. We will use a theorem from Wilton \cite[Theorem 18]{Wil2012oneended}, for which we slightly correct the formulation:

\begin{lemma}\label{wilton}
     Let $\Gamma$ be a graph of groups with infinite cyclic edge groups and a finitely generated fundamental group $L$. Suppose every vertex group has rank at least 2 or, if it is cyclic, then the vertex has exactly one incident edge and the inclusion map of the edge group into that vertex is an isomorphism. Then $L$ is one-ended if and only if every vertex group in $\Gamma$ is freely indecomposable relative to the incident edge groups.
\end{lemma}

Note that free groups are not one-ended (they have infinitely many ends).

\begin{proposition}
    Let $L<M$ be free groups. Then $L$ is a free factor of $M$ if and only if $M$ embeds in $L_n=\ang{L,t_1,\ldots,t_n\mid[c_i,t_i]=1}$, where $c_1,\ldots,c_n\in L$ are primitive and distinct, such that the embedding of $L$ along $M\hookrightarrow L_n$ equals $L$.
\end{proposition}
\begin{proof}
    ($\Rightarrow$) This follows immediately from Remark~\ref{remark-ice}.

($\Leftarrow$) Consider $L_n$ as the fundamental group of a graph of groups with a single vertex group $L$ and $n$ loops. $M$ acts on the corresponding Bass-Serre tree, inducing a graph of groups $\Gamma$ with fundamental group $M$. We will use induction on the rank of $M$. Consequently, we need only consider the connected component of $\Gamma$ containing $L$.

By Lemma~\ref{wilton}, there is a vertex group $G_v$ in $\Gamma$ that is freely decomposable relative to its incident edge groups. If $G_v\ne L$ and $v$ is a cut-point of $\Gamma$, then we apply the induction hypothesis to the connected component containing $L$. If $G_v\ne L$ and $v$ is not a cut-point, replace $v$ with two vertices, one for each factor of $G_v$, and an edge with trivial edge group $G_e$ between them. Then $M$ is a free product of $\pi_1(\Gamma-e)$ and the stable letter corresponding to $G_e$, so we apply the induction hypothesis to $\pi_1(\Gamma-e)$, which has lower rank.
 
Suppose $G_v=L$ and no other vertex group is freely decomposable relative to its edge groups. Let $L=A*B$ and suppose $c_1,\ldots,c_k$ are conjugate into $A$ and $c_{k+1},\ldots,c_n$ are conjugate into $B$. Consider another vertex group $L^x\cap M$, where $x\in L_n$. By Bass-Serre theory, we have
$$L^x\cap M=A_0^x*\cdots*A_p^x*B_0^x*\cdots*B_q^x*F,$$
where $A_j$ is conjugate into $A$ for all $j$, $B_j$ is conjugate into $B$ for all $j$, and $F$ is free. Let $A'=A_0^x*\cdots*A_p^x$ and let $B'=B_0^x*\cdots*B_q^x$. We claim that either $L^x\cap M=A'$ or $L^x\cap M=B'$. If $F$ is nontrivial then $L^x\cap M$ is freely decomposable with all incident edge groups conjugate into $A'*B'$, contradicting our assumption. The situation is similar if there are no incident edge groups conjugate into $A'$ or no edge groups conjugate into $B'$. Suppose both $A'$ and $B'$ contain conjugates of edge groups. Every edge group is of the form $\ang{c_i^{x_i}}$ for some $x_i\in L_n$. If $i\le k$, $c_i$ is conjugate into $A$, so $c_i^{x_i}$ must be conjugate into $A'$. Otherwise $c_i^{x_i}$ is conjugate into $B'$. So $L^x\cap M$ is freely decomposable relative to its edge groups, contradicting our assumption.

Therefore, every vertex group other than $L$ is either of the form $A_0^x*\cdots*A_p^x$ or of the form $B_0^x*\cdots*B_q^x$. Furthermore, since $c_i$ is conjugate into $A$ if and only if $i\le k$, there is no edge connecting a vertex of the form $A_0^x*\cdots*A_p^x$ to one of the form $B_0^x*\cdots*B_q^x$. So removing the trivial edge between $A$ and $B$ disconnects $\Gamma$ into two components. Then $M$ is the free product of freely indecomposable groups and hence not free.
\end{proof}

\section{Countable elementary free groups}

Finitely generated elementary free groups were described by Kharlampovich and Myasnikov as regular NTQ groups and Sela as hyperbolic fully residually free towers (see \cite{KhaMyaSkl2020fraisse}). In this section, we will consider some examples of countable non-finitely generated elementary free groups and will give a negative answer to a question in \cite{KhaMyaSkl2020fraisse}: Are all countable elementary free groups obtained as the union of a chain of finitely generated elementary free groups? (Note that by a chain, we mean a chain with order type $\omega$, i.e., a sequence of groups $G_0\le G_1\le\cdots\le G_n\le\cdots$ where $n<\omega$.) 

Notice that  every countable universally free group is a union of a chain of limit groups.
Indeed, a group is universally free if and only if it is locally a fully residually free group (every finitely generated subgroup is a limit group), see, for example  \cite[Theorem~5.9]{Chi2001introduction}.

\begin{theorem} A free product of abelian groups that are each elementarily equivalent to $\mathbb Z$ is an elementary free group.
\end{theorem}

This follows from \cite[Theorem 7.1]{Sel2010diophantine}, which states that for groups $A_1,B_1,A_2,B_2$, if $A_1$ is elementarily equivalent to $A_2$ and $B_1$ is elementarily equivalent to $B_2$, then $A_1\ast B_1$ is elementarily equivalent to $A_2\ast B_2$.

We now recall some of Szmielew's results on torsion-free abelian groups (see \cite{EklFis1972elementary}). Let $A$ be a torsion-free abelian group. Define $\alpha_p(A)=\dim(A/pA)$ over the field of $p$ elements, if it is finite, and $\alpha_p=\infty$ otherwise. For example, for any prime $p$, we have $\alpha_p(\mathbb Z)=1$. The Szmielew characteristic of $A$ is $\psi(A)=(\alpha_2(A), \alpha_3(A),\alpha_5(A), \ldots)$. Then for a torsion-free abelian group $B$, $\Th(A)=\Th(B)$ if and only if $\psi (A)=\psi (B)$. In particular, if $C$ is divisible, then $\Th(A)=\Th(A\oplus C)$, e.g., $\Th(\mathbb Z)=\Th(\mathbb Z\oplus \mathbb Q)$.  

Given a group $G$, $\dim (A/pA)<2$ for any abelian subgroup $A\le G$ if and only if
$$G\models\forall x_1,x_2 \exists y \prn{[x_1,x_2]=1\rightarrow  \bigvee_{(m_1,m_2)\in S}x_1^{m_1}x_2^{m_2}=y^p},$$
where $S$ is the set of all non-trivial tuples $(m_1,m_2)$ where $0\leq m_i<p.$

\begin{theorem} The elementary free group $T=\mathbb Z\ast(\mathbb Z\oplus \mathbb Q)$ cannot be represented as a union of a chain of finitely generated  elementary free groups.\end{theorem}
\begin{proof}
    Any chain of finitely generated groups whose union is $\Z\oplus\Q$ must at some step be isomorphic to the free abelian group of rank 2. So a chain whose union is $T$ must at some step include a non-cyclic abelian subgroup. But no finitely generated elementary free group can contain a non-cyclic abelian subgroup.
\end{proof}

\printbibliography

\end{document}